\documentclass[10pt]{article}
\usepackage[margin=1.5cm]{geometry}
\usepackage{amsmath,amsthm,amsfonts,amssymb,amscd}
\usepackage{cite,graphicx,fdsymbol,fancyhdr,hyperref,utfsym}
\usepackage[dvipsnames]{xcolor}

\newcommand{\club}{\clubsuit}
\newcommand{\spade}{\spadesuit}
\newcommand{\heart}{\color{red}\varheartsuit\color{black}}
\renewcommand{\diamond}{\color{red}\vardiamondsuit\color{black}}

\newtheorem{definition}{Definition}[section]
\newtheorem{proposition}{Proposition}[section]
\newtheorem{lemma}{Lemma}[section]

\title{\bf MinDist is less than 7}

\author{
Purushottam Saha\thanks{Corresponding author: \texttt{puru8017@gmail.com}} \\
Indian Statistical Institute, Kolkata
\and
Diganta Mukherjee \\
Indian Statistical Institute, Kolkata
}

\date{} 

\begin{document}
\maketitle

\begin{abstract}
The metric MinDist, introduced recently to quantify the distance of an arbitrary Rummy hand from a valid declaration, plays a central role in algorithmic hand evaluation and optimal play. Existing results show that the MinDist of any $13$-card Rummy hand from a single deck is bounded above by $9$. In this paper, we sharpen this bound and prove that the MinDist of any hand is at most $7$. We further show that this bound is tight by explicitly exhibiting a hand whose MinDist equals $7$ for a suitable choice of wildcard joker. The proof combines elementary combinatorial arguments with structural properties of card partitions across suits and resolves the gap between the previously known upper bound and the true extremal value.
\end{abstract}

\section{Introduction}

Rummy is a widely played card game in which a player aims to partition a hand of $13$ cards into valid \emph{melds}, consisting of sequences and sets, subject to the presence of a designated wildcard joker. Beyond its recreational importance, the game admits a rich combinatorial structure that has recently motivated formal mathematical and simulation based analysis \cite{Megarry1975RummyReview,Finkle2017Rummy}, particularly in the context of hand evaluation and algorithmic decision-making\cite{SahaRummyMetric2025}.

A central concept in this direction is the metric \emph{MinDist}\cite{SahaRummyMetric2025}, which measures how far a given hand is from being declarable, quantified as the minimum number of card replacements required to transform the hand into a valid declaration. This metric provides a principled way to rank hands, design heuristics for automated play, and analyze the difficulty of achieving a valid configuration under partial information.
Previous work established that the MinDist of any Rummy hand is bounded above by $9$, a simple result obtained via general counting arguments \cite{SahaRummyMetric2025}. While this bound suffices for algorithmic guarantees, empirical evidence shows the number is between 2 to 4 for a randomly generated hand with high probability. Through our extensive adversarial search, we could only find hands with maximal possible MinDist of $7$. This leaves open the natural question: \emph{what is the smallest upper bound on MinDist}? Determining the exact value of this bound is not only of theoretical interest but also clarifies the worst-case behavior of hand configurations.

The purpose of this short paper is to resolve this question. We show that the MinDist of any $13$-card Rummy hand is at most $7$, thereby improving the known bound by two. Moreover, we demonstrate that this bound is tight by explicitly constructing a hand (and an wild-card joker) that attains MinDist $7$. Our arguments rely on elementary combinatorial reasoning about suit distributions, gaps between card values, and forced structural constraints, supplemented by limited computational verification in a (comparatively) small residual case.
The results presented here close the gap between known upper bounds and achievable extremal configurations, and provide a precise characterization of the maximal distance of a Rummy hand from a valid declaration.

\section{Main Results}\label{sec-2}
The game of Rummy is played with a single deck (cards $A\;2\;...\;10\;J\;Q\;K$ for each of the four suits) and a printed joker, 13 cards are distributed to each player, and a separate card is considered the wild-card joker. A pile is set up by placing the top most card from the remaining deck. The game continues by each player taking turn to draw a card from the closed deck (face down) or the pile (face up, but only last dropped card is allowed to be picked), and drop a card (the card chosen from pile is not allowed to be dropped on the same turn) to the pile, till one player declares their hand as a partition of valid melds. Each valid meld is one of the following:\\
\begin{itemize}
    \item Pure Sequence: (at least 3) Cards of same suit, which are consecutive in $A\;2\;3\;4\;5\;6\;7\;8\;9\;10\;J\;Q\;K\;A$ order, e.g. $3\heart\;4\heart\;5\heart$ is a valid Pure Sequence, but 
    \item Impure Sequence: Same as a Pure Sequence, though one or more cards can be replaced by a joker card.
    \item Pure Set: (at least 3) Cards of same value, and different suits.
    \item Impure Set: Same as a Pure Set, though one or more cards can be replaced by a joker card.
\end{itemize}
A criteria that is often imposed on a valid declaration, is that at least one of the melds need to be a Pure Sequence, and another one of the melds need to be an Pure/Impure Sequence.

The MinDist metric, introduced in \cite{SahaRummyMetric2025}, quantifies how far a given hand is from being declarable.
\begin{definition}
    MinDist of a given hand $h$, for a given wild-card joker $wcj$, is defined as $$\emph{MinDist}(h,wcj) = \min_{h'\in declarable\;hands}d_{wcj}(h,h'),$$ where $d_{wcj}(h,h')$  is the distance between 2 hands $h$ and $h'$, i.e. the minimum number of cards that needs to be replaced of $h$, to create $h'$.
\end{definition}
This metric may also be interpreted as the minimum number of successful card replacements required to reach a declarable hand, making it a natural guide for algorithmic play and hand evaluation\cite{SahaRummyMetric2025}.\\

A general upper bound was established in \cite{SahaRummyMetric2025}.
\begin{proposition}\label{prop-1}
MinDist of any given hand is less than $9$.
\end{proposition}
\begin{proof}
   We present an alternate proof of the result. As there are $4$ suits and $13$ total cards dealt, by simple PHP arguments there is at least one suit with $4$ cards. If we can replace (at most) all $9$ other cards to get the complete suit, we can create two pure sequences, say, $2\;3\;4$ and $5\;6\;7$ (of that suit) and all the rest cards ($8$ to $A$) as another large pure sequence, hence a valid meld partition. So at most $9$ cards are required to be replaced, hence MinDist of any hand is less than $9$. 
\end{proof}

We now strengthen the bound, in the following propositions.
\begin{proposition}\label{prop-2}
MinDist of any given hand is less than $8$.
\end{proposition}

\begin{proof} 
Let us first note that, amongst the $13$ different value of cards (across $4$ suits), one of them is assigned as a wild-card joker, which can be used as any other card, hence w.l.o.g. we can assume we do not obtain a card which is a wild-card joker. Hence, as there are now $12$ possible values, and $13$ cards dealt, there must be $2$ cards which has the same value. We wish to form a valid set by expecting to receive the two other cards of the same value from the other suits (hence we are utilizing $2$ of the given $13$ cards). As we are already forming a $4$-length set, we need to arrange $9$ more cards, and that can be done by $3$ sequences\footnote{one of them must be pure by requirement, but we can always make them pure in these calculations as we can obtain any card we like to replace our cards}, guided by any $3$ cards from the rest of the hand (and finding cards near those). Hence, we are able to utilize at least $2$ (from the set) + $3$ (from $3$ sequences) $= 5$ cards from our hand of size $13$, which means replacing at most $13-5 = 8$ cards, we can turn any hand of rummy into a complete hand, i.e. MinDist of any given hand is at most $8$.
\end{proof}

To proceed further, we introduce a notion of distance within a suit.
\begin{definition}
    Distance between two distinct cards of same suit (also referred as 'gap') is the minimum number of cards between the cards in the order $A\;2\;3\;4\;5\;6\;7\;8\;9\;10\;J\;Q\;K\;A$. Observe that this is also the minimum number of cards required to create a valid sequence involving the given cards, if the gap is at least $1$. For example, gap between $2$ and $3$ is $0$, but gap between $J$ and $K$ is $1$.
\end{definition}
The following lemma captures a key structural constraint.
\begin{lemma}\label{lemma1}
    If a suit has 4 or more cards (of 11 possibilities for values), then the maximum possible minimum distance between any 2 cards is 2.
\end{lemma}
\begin{proof}
    Allow circular sequences (i.e. $K\spade A\spade 2\spade$ is also now a valid sequence), which can only reduce gaps. If the minimum gap between cards were at most $3$, then $4$ gaps along with $4$ cards need to be placed, which will take at least $12+4=16$ distinct possible values. Hence the largest possible minimum gap between cards is $2$. Hence the result follows.
\end{proof}

We now arrive at the main result.
\begin{proposition}
    \begin{enumerate}
        \item[(a)] MinDist of any given hand of Rummy is at most 7.
        \item[(b)] This bound is tight, as an example of a hand with MinDist = 7 is: $2\heart\;3\club\;4\spade\;5\diamond\;6\heart\;7\club\;8\spade\;9\diamond\;10\heart\;J\club\;Q\spade\;K\spade\; K\diamond$, with $A\heart$ as wildcard joker.
    \end{enumerate}
\end{proposition}

\begin{proof}
    We start as with the last proof, that we observe 2 cards with same values but of different suit, and wish to form a valid set out of the cards. Considering the rest 11 cards, we observe that they come from 11 different possible values, and if any two cards are of same value, then we are done, as we would use them to form another set of size 3, (together with another set that was created earlier, the case where we utilize 3 more cards to form 3 sequences, hence utilizing at least 7 cards in total, i.e. MinDist bound becomes 6), hence we assume w.l.o.g. the cards are all of different values. 

    Lemma ~\ref{lemma1} shows that if there are 4 cards in a suit, then we will obtain two cards with at most gap 2. We can use them to create a sequence of length 4, together with the past valid set formation (we can use one of length 3), we have used 2+2 cards to create valid melds. Alongside, to fill the rest 6 card gap, we can utilize 2 more cards to create 2 more valid sequences, hence we can use 6 of our 13 cards, which proves the maximum possible MinDist is 7 for this case.

    The only other part left is if the partition of the rest 11 cards is done by the lengths 3-3-3-2. Till now, we have not obtained a clear mathematical proof for this case, but the cases are now fairly easy to check with a computer. \\
    We start with a set of 1 to 12, and partition it into lengths of 3-3-3-2 with leaving 1 card, which will be the card that we have the same value of for 2 suits. If we can show a result similar to Lemma ~\ref{lemma1}, that maximum possible minimum distance between 2 cards is at most 3 for these partitions, and there are at least two such pairs with minimum distance 3 between each pair, then we are done, as the rest of the proof follows from the fact that those pairs can be utilized to create two 5-long sequences, and a 3-long set, by utilizing a total of $2+2+2=6$ cards from our hand of 13. The total count of all such partitions are less than $\binom{12}{1}\times \binom{11}{2}\times \binom{9}{3}\times \binom{6}{3}\times \binom{3}{3} = 1108800$, which is fairly easy to loop over with a fairly fast computer program. The code used to check all the cases can be found \href{https://github.com/purushottam-saha/MinDist-less-than-7}{here}.

    The result mentioned actually is a tight one. The hand given as example is exactly such a hand, i.e. distances between any 2 cards of the same suit, (if the two $K$ are ignored, o.w. one can proceed with $Q\spade K\spade$, and replace the $Q$ by the $K$), is at least 3, and there are more than 2 pairs which satisfy this. A target valid declaration for the hand is $2\heart\;\usym{1F0A0}\;\usym{1F0A0}\;\usym{1F0A0}\;6\heart\;\;3\club\;\usym{1F0A0}\;\usym{1F0A0}\;\usym{1F0A0}\;7\club\;\;K\spade\;K\diamond\;\usym{1F0A0}$ , from which it is obvious that the MinDist is at most 7, and to show that the hand achieves MinDist of exactly 7, we use the computer program given in \cite{SahaRummyMetric2025}, which computes the MinDist and confirms the claim. 
    
\end{proof}

\section{Conclusion}

In this paper, we improved the previously known universal upper bound on the MinDist of a Rummy hand from $9$ to $7$. Using combinatorial arguments based on suit distributions, gap constraints, and meld construction, we showed that no hand can be farther than seven card replacements from a valid declaration. We further established the tightness of this bound by presenting an explicit hand whose MinDist equals $7$ for a suitable wildcard joker.

This result completes the extremal analysis of the MinDist metric and settles an open gap in earlier work. Beyond its theoretical significance, the exact bound has practical implications for algorithmic hand evaluation, heuristic design, and worst-case analysis in automated Rummy play. Future work may explore analogous distance metrics under rule variations, multiple jokers, or alternative meld constraints, as well as tighter structural characterizations of extremal hands.
\bibliographystyle{abbrv}

\end{document}